\documentclass[12pt]{amsart}

\usepackage[margin=1.08in]{geometry}
\usepackage{amsmath,amssymb,amsthm,mathtools}
\usepackage[colorlinks=true,linkcolor=blue,citecolor=blue,urlcolor=blue]{hyperref}
\usepackage{microtype}

\usepackage[nodisplayskipstretch]{setspace}
\theoremstyle{plain}
\newtheorem{theorem}{Theorem}[section]
\newtheorem{proposition}[theorem]{Proposition}
\newtheorem{lemma}[theorem]{Lemma}
\newtheorem{corollary}[theorem]{Corollary}
\theoremstyle{definition}
\newtheorem{definition}[theorem]{Definition}
\newtheorem{convention}[theorem]{Convention}
\theoremstyle{remark}
\newtheorem{remark}[theorem]{Remark}

\newcommand{\PP}{\mathbb P}
\newcommand{\CC}{\mathbb C}
\newcommand{\ZZ}{\mathbb Z}
\newcommand{\cE}{\mathcal E}
\newcommand{\cN}{\mathcal N}
\newcommand{\cD}{\mathcal D}
\newcommand{\cM}{\mathcal M}
\newcommand{\resz}{\operatorname*{res}_{z}}
\newcommand{\Aut}{\operatorname{Aut}}
\newcommand{\Part}{\operatorname{Part}}
\newcommand{\eps}{\varepsilon}
\newcommand{\del}{\delta}
\newcommand{\varsig}{\varsigma}

\numberwithin{equation}{section}

\title[Negative contacts and Toda]{On the Toda hierarchy for the stationary theory of the tube with negative contact orders}

\author{Hsian-Hua Tseng}
\address{Department of Mathematics\\ The Ohio State University\\ 100 Math Tower, 231 West 18th Avenue\\ Columbus, OH 43210\\ USA}
\email{hhtseng@math.ohio-state.edu}

\date{\today}

\hypersetup{
  pdftitle={On the Toda hierarchy for the stationary theory of the tube with negative contact orders},
  pdfauthor={Hsian-Hua Tseng}
}

\begin{document}

\begin{abstract}
We construct a master $2$-Toda tau sequence for the stationary relative theory of $(\PP^1,0,\infty)$ with any fixed finite lists of negative contacts.  The negative-contact blocks are encoded by finite products of exponentials of infinite-wedge Lie algebra elements, with the block at $\infty$ taken as the adjoint of the entire corresponding left block.  After adjoining independent positive-profile Toda times, the resulting charged-vacuum matrix elements form a genuine $2$-Toda tau sequence.  Following the standard $Q,u$ specialization, its positive signed-degree coefficients recover the stationary relative Gromov--Witten invariants with the prescribed positive and negative contacts, which occur in and are constrained by the induced coefficientwise Hirota equations.

The master tau sequence is larger than the enumerative sector: it generally contains nonzero, non-vacuum signed-degree-zero coefficients, as well as auxiliary nonzero lattice components, for which no enumerative interpretation is asserted.  Consequently the Hirota equations form a coupled master system and do not, without additional input, close on the positive-degree invariants alone.  The construction uses the large-root definition of negative contact, the Johnson--Kumaran--Wu operator formula for the tube, and the Kyoto/Pl\"ucker construction of the $2$-Toda hierarchy.
\end{abstract}

\maketitle
\tableofcontents

\section{Introduction and scope}

Let $\omega\in H^2(\PP^1)$ be the point class.  We consider the disconnected non-equivariant stationary relative theory of the {\em tube}\footnote{The {\em cap} geometry is $(\mathbb{P}^1,0)$ or $(\mathbb{P}^1,\infty)$. }
\[
    (\PP^1,0,\infty),
\]
with relative contact orders allowed to be any nonzero integers.  Positive contact orders are ordinary relative tangency orders.  Negative contact orders are understood in the large-root sense of Fan--Wu--You, as recalled and used by Kumaran--Wu.  Thus a negative contact $b<0$ at $0$ is represented on the $r$th root stack by the twisted sector of age $(r+b)/r$, and the invariant is obtained after multiplying by one factor of $r$ per negative contact and extracting the constant term in $r$; for the tube one performs this construction at both relative points and extracts the $r^0s^0$ coefficient.  See \cite[Defs.~2.1--2.3 and Sec.~2.2]{KW}, based on \cite{FWY20,FWY21}.

The operator formula for negative-contact invariants in the stationary non-equivariant theory is obtained by Kumaran--Wu from Johnson's orbifold operator formula \cite{JohnsonThesis,Johnson,KW}. The main theorem on integrable hierarchy, Theorem~\ref{thm:main}, is stated in this setting.  It does not assert that a fixed negative-contact coefficient is itself a scalar tau-function.  Instead, after a finite system of nilpotent variables is introduced, the positive-degree chosen-contact generating functions occur among the coefficients of one genuine $2$-Toda tau sequence and satisfy the coefficient equations induced from that larger system.

Kumaran--Wu state the negative-contact cap and tube formulae for positive signed degree $d>0$, and every non-vacuum enumerative identification in this note is made in that range. At $0$, write a relative condition as
\[
  (\lambda;\mathbf b),
  \qquad
  \lambda\text{ a positive partition},\quad
  \mathbf b=(b_1,\ldots,b_m),\quad b_i<0.
\]
At $\infty$, write
\[
  (\rho;\mathbf c),
  \qquad
  \rho\text{ a positive partition},\quad
  \mathbf c=(c_1,\ldots,c_{m'}),\quad c_j<0.
\]
The signed degree condition is
\begin{equation}
\label{eq:signed-degree}
  d=|\lambda|+\sum_{i=1}^m b_i
   =|\rho|+\sum_{j=1}^{m'} c_j>0.
\end{equation}

The empty-profile stationary degree-zero vacuum is also identified and has constant term $1$ in the descendant variables.  The master tau sequence nevertheless contains further algebraic coefficients of signed degree zero, which can be nonzero.  We retain them as auxiliary sectors because projecting them away need not preserve group-likeness or the bilinear identities; no relative-invariant interpretation of those sectors is claimed.

The main theorem uses the positive-profile variables $\mathbf p=(p_1,p_2,\ldots)$ and $\mathbf q=(q_1,q_2,\ldots)$.  Ordered positive contacts are recovered by labelled specialization; see Corollary~\ref{cor:chosen-ordered}.

The main point is not merely that a fermionic matrix element is a Toda tau-function. The nontrivial issue is to construct a single group-like operator simultaneously encoding prescribed negative contacts, with the correct ordering and adjoint structure, and to identify precisely which coefficients recover the negative-contact relative invariants.

\subsection{Generative-AI disclosure}
{\em This paper, generated by GPT-5.5 and 5.6 Sol Pro, represents an outcome of the human author's ongoing experiment whose grand goal is to see how the human author's research in mathematics can benefit from generative AI.

As found in general in \cite{FWY20,FWY21} and carried out explicitly in \cite{KW}, relative Gromov-Witten invariants of the tube $(\mathbb{P}^1, 0,\infty)$ with possibly negative contact orders can be obtained from orbifold Gromov-Witten invariants of the orbifold $\mathbb{P}^1_{r,s}$ for $r, s$ sufficiently large. As one can deduce results on integrable hierarchy from the operator formula for Gromov-Witten invariants of $\mathbb{P}^1_{r,s}$ proven in \cite{JohnsonThesis,Johnson}, it is natural to try to obtain results on integrable hierarchy for relative Gromov-Witten theory of the tube with negative contact orders from this ``sufficiently large $r,s$'' consideration. There is no doubt that this idea was known to experts: it was known to the human author for several years.  

It was expected that carrying out this idea requires some manipulations of equations. In view of the recent advance in generative AI, the human author decided to ask GPT-5.5 Pro to carry out this idea. The present paper is the outcome of several rounds of interactions with GPT-5.5 and 5.6 Sol Pro. The human author reviewed the results but made minimal edits to the paper.}


\subsection{Acknowledgment}
We thank F. You for corrections.

\section{Operator formulae}

\subsection{The infinite wedge}

Let $V$ be the vector space with basis indexed by half-integers, and let
\begin{equation}
\label{eq:charged-Fock}
  \mathcal F=\bigoplus_{n\in\ZZ}\mathcal F^{(n)}
\end{equation}
be the full charged infinite wedge.  Write $\lvert0\rangle=v_\emptyset$ for the charge-zero vacuum.  If $T$ is the standard translation operator, set
\begin{equation}
\label{eq:charged-vacua}
  \lvert n\rangle=T^n\lvert0\rangle,
  \qquad
  \langle n\rvert=\langle0\rvert T^{-n}.
\end{equation}
For a charge-preserving operator $A$, write $\langle A\rangle=\langle0\rvert A\lvert0\rangle$.  Let $\psi_i$ and $\psi_j^*$ be the standard fermionic creation and annihilation operators, and define the normally ordered matrix units by
\begin{equation}
\label{eq:matrix-units}
  E_{ij}:=:\!\psi_i\psi_j^*\!:\,.
\end{equation}
The following standard operators act charge-preservingly on $\mathcal F$:
\begin{equation}
\label{eq:E-def}
 \cE_a(z)=
 \sum_{k\in\ZZ+\frac12} e^{z(k-a/2)}E_{k-a,k}
 +\frac{\delta_{a,0}}{e^{z/2}-e^{-z/2}}.
\end{equation}
They satisfy
\begin{equation}
\label{eq:E-comm}
 [\cE_a(z),\cE_b(w)]
 =\varsig(aw-bz)\cE_{a+b}(z+w),
 \qquad
 \varsig(z)=e^{z/2}-e^{-z/2}.
\end{equation}
The Heisenberg operators are $\alpha_k=\cE_k(0)$ for $k\ne0$, and
\begin{equation}
\label{eq:heisenberg}
 [\alpha_k,\alpha_l]=k\delta_{k+l,0}.
\end{equation}
We expand
\begin{equation}
\label{eq:E-coeff}
  \cE_a(z)=\sum_{r\in\ZZ}\cE_a[r]z^{r+1}.
\end{equation}
With respect to the standard wedge inner product,
\begin{equation}
\label{eq:E-adjoint}
  \cE_a[r]^\dagger=\cE_{-a}[r].
\end{equation}
Throughout, $\dagger$ denotes the algebraic adjoint anti-involution on the completed infinite-wedge operator algebra:
\begin{equation}
\label{eq:algebraic-adjoint}
  (AB)^\dagger=B^\dagger A^\dagger,
  \qquad
  \alpha_k^\dagger=\alpha_{-k}.
\end{equation}
It fixes the formal scalar coefficient ring---in particular $Q$, $u$, and, when they are introduced below, the nilpotent variables $\eps_i,\del_j$---and is not complex conjugation of formal coefficients.  Thus $\cE_0[k]$ is the stationary descendent operator corresponding to $\tau_k(\omega)$, and $\cE_b[0]$ with $b<0$ is the negative-contact operator used by Kumaran--Wu.  These conventions are those of Okounkov--Pandharipande and Kumaran--Wu \cite[Sec.~2]{OPcompleted}, \cite[Sec.~3]{KW}.

The energy operator
\begin{equation}
\label{eq:energy-operator}
  H=\sum_{k\in\ZZ+\frac12}kE_{kk}
\end{equation}
satisfies
\begin{equation}
\label{eq:energy-commutator}
  [H,\cE_a[r]]=-a\cE_a[r].
\end{equation}
The charge-zero vacuum has energy $0$, and the spectrum of $H$ on $\mathcal F^{(0)}$ is nonnegative.  These facts will be used to make the signed-degree vanishing explicit.

The translation operator satisfies
\begin{equation}
\label{eq:T-identity}
   T^{-n}\cE_a(z)T^n=e^{nz}\cE_a(z),
\end{equation}
including the scalar term in $\cE_0(z)$.  In particular, the central anomaly in the normally ordered diagonal operator is exactly what makes \eqref{eq:T-identity} true for $a=0$.

\subsection{Kumaran--Wu's negative-contact tube formula}

Let
\[
  \vec\mu_0=(a_1,\ldots,a_{\ell_0},b_1,\ldots,b_{m_0}),
  \qquad
  \vec\mu_\infty=(a'_1,\ldots,a'_{\ell_\infty},b'_1,\ldots,b'_{m_\infty}),
\]
where $a_i,a'_j>0$ and $b_i,b'_j<0$.  We use the conventions
\begin{equation}
\label{eq:empty-block-conventions}
  \Part(\varnothing)=\{\varnothing\},
  \qquad
  \prod_{x\in\varnothing}(\cdots)=1.
\end{equation}
Thus a product indexed by an empty set or by the unique empty set partition is $1$; in particular, $\mathcal B_{\varnothing}=1$.  For a nonempty ordered subset $B=\{j_1<\cdots<j_r\}$ of the negative labels, set
\begin{equation}
\label{eq:block-weight}
  b_B=\sum_{j\in B}b_j,
  \qquad
  N_B=(r-1)!\prod_{h=2}^{r}b_{j_h}.
\end{equation}
For a set partition $P$ of the negative labels, ordered by increasing minima of blocks, define
\begin{equation}
\label{eq:B-operator}
  \mathcal B_{\mathbf b}
  =\sum_{P\in\Part([m_0])}
    \left(\prod_{B\in P}N_B\right)
    \prod_{B\in P}^{\longrightarrow}\cE_{b_B}[0].
\end{equation}
For the right-hand negative list $\mathbf c=(c_1,\ldots,c_{m_\infty})$, define for every nonempty ordered subset $C=\{j_1<\cdots<j_r\}$
\[
  c_C=\sum_{j\in C}c_j,
  \qquad
  N_C=(r-1)!\prod_{h=2}^{r}c_{j_h}.
\]
Let $\mathcal B_{\mathbf c}$ denote the operator obtained from \eqref{eq:B-operator} by replacing $\mathbf b$ with $\mathbf c$.  The right-hand block in the tube pairing is the adjoint of this \emph{entire} operator:
\begin{equation}
\label{eq:Bdagger-operator}
  \mathcal B^\dagger_{\mathbf c}
  :=\left(\mathcal B_{\mathbf c}\right)^\dagger
  =\sum_{P'\in\Part([m_\infty])}
    \left(\prod_{C\in P'}N_C\right)
    \prod_{C\in P'}^{\longleftarrow}\cE_{c_C}[0]^\dagger.
\end{equation}
Here the blocks of $P'$ are first listed in increasing order of their minima; the left-pointing product means that the corresponding adjoint factors occur in the reverse order.  Thus, if the blocks are $C_1,\ldots,C_s$ with increasing minima, the product in \eqref{eq:Bdagger-operator} is
$\cE_{c_{C_s}}[0]^\dagger\cdots\cE_{c_{C_1}}[0]^\dagger$.
This is also the order dictated by the cited source: \cite[Defs.~4.10--4.11]{KW} first orders the blocks by increasing minima and defines $\cE_{\mathbf c_{P'}}[0]$ as the resulting ordered product, while \cite[Thm.~4.18]{KW} inserts $\cE_{\mathbf c_{P'}}^{*}[0]$ at $\infty$.  The star is therefore applied to the composite ordered block, and adjunction reverses its factors.  Factorwise adjunction without this reversal is not the whole adjoint and is not generally invariant under relabeling.  Equation~\eqref{eq:Bdagger-operator} spells out the source convention used in every enumerative statement below.

Kumaran--Wu prove the stationary tube formula
\begin{align}
\label{eq:KW}
&\left\langle \vec\mu_0\,\middle|\,
    \prod_{i=1}^n\tau_{k_i}(\omega)
  \,\middle|\,\vec\mu_\infty\right\rangle^{\bullet}
=
\frac{1}{\prod_{i=1}^{\ell_0}a_i\prod_{j=1}^{\ell_\infty}a'_j}
\left\langle
 \prod_{i=1}^{\ell_0}\alpha_{a_i}\,
 \mathcal B_{\mathbf b}\,
 \prod_{i=1}^{n}\cE_0[k_i]\,
 \mathcal B^\dagger_{\mathbf c}\,
 \prod_{j=1}^{\ell_\infty}\alpha_{-a'_j}
\right\rangle.
\end{align}
This is the stationary tube identity of \cite[Thm.~4.18]{KW}, with its starred composite right block expanded as the whole adjoint; the cap analogue is \cite[Thm.~4.16]{KW}.  Writing the order explicitly in \eqref{eq:Bdagger-operator} removes any ambiguity that may be hidden by a star notation for the right block.  When there are no negative contacts, \eqref{eq:KW} specializes to the Okounkov--Pandharipande stationary tube formula \cite[Sec.~3]{OPcompleted}.  The algebraic tau construction below is independent of the enumerative provenance of \eqref{eq:KW}, while the coefficient identification uses precisely the displayed whole-adjoint convention.

\begin{convention}[Genus and Toda dispersion]
\label{conv:u}
The published formula \eqref{eq:KW} is written in the usual genus-suppressed stationary convention.  For a possibly disconnected source curve $C$, define the virtual genus index $g$ by
\begin{equation}
\label{eq:disconnected-genus}
  2g-2=-\chi(C),
\end{equation}
so $g$ need not be a nonnegative connected genus.  We retain a formal Toda dispersion parameter $u$ by the homogeneous rescaling
\begin{equation}
\label{eq:u-rescaling}
  \frac{\alpha_{a}}{a}\mapsto \frac{\alpha_a}{ua},
  \qquad
  \sum_{k\ge0}x_k\cE_0[k]\mapsto \sum_{k\ge0}x_ku^k\cE_0[k],
\end{equation}
which is equivalent to representing a stationary variable $z$ by $u^{-1}\cE_0(uz)$.  Setting $u=1$ recovers \eqref{eq:KW}.  The hierarchy proof itself is formal in $u$ and remains valid after setting $u=1$.
\end{convention}

\begin{lemma}[Restoring the genus variable]
\label{lem:u-restore}
With the conventions of Kumaran--Wu's large-root calculation before the final specialization $u=1$, the stationary negative-contact tube formula \eqref{eq:KW} has the following homogeneous $u$-refinement.  Each positive relative part $a$ contributes $\alpha_a/(ua)$ on the left and $\alpha_{-a}/(ua)$ on the right; each stationary descendent $\tau_k(\omega)$ contributes $u^k\cE_0[k]$; and the negative-contact block operators $\mathcal B_{\mathbf b}$ and $\mathcal B^\dagger_{\mathbf c}$ are unchanged.  The coefficient of a fixed monomial in the stationary variables is then the disconnected genus series $\sum_g u^{2g-2}\langle\cdots\rangle^\bullet_{g,d}$, with $g$ defined by \eqref{eq:disconnected-genus}.
\end{lemma}

\begin{proof}
Kumaran--Wu retain Johnson's genus variable through the large-root calculation and specialize to $u=1$ only at the end; see the reductions leading to their equations (37)--(39) and Theorems~4.16 and 4.18 in \cite{KW}.  In that calculation a positive relative marking $a>0$ is produced by the leading term $\alpha_a/(ua)$, while a stationary insertion is represented before coefficient extraction by $u^{-1}\cE_0(uz)$, whose coefficient of $z^{k+1}$ is $u^k\cE_0[k]$.

There is no residual power of $u$ in a negative block.  For a block $B$, coefficient extraction contributes $u^{|B|}\cE_{b_B}[0]$, and the large-root prefactor contributes one factor $u^{-1}$ for each of the $|B|$ negative markings.  These powers cancel block by block.  The same cancellation holds at $\infty$, and adjunction does not change the power of $u$.

The virtual dimension gives an independent homogeneity check.  Suppose a connected component $C_0$ carries $N$ stationary markings, $\ell_0$ positive and $m_0$ negative relative markings over $0$, and $\ell_\infty$ positive and $m_\infty$ negative relative markings over $\infty$.  Write the corresponding contact orders as $a_i,b_i$ over $0$ and $a'_j,c_j$ over $\infty$, and denote the signed degree of this component by
\[
  d_{C_0}=\sum_{i=1}^{\ell_0}a_i+\sum_{i=1}^{m_0}b_i
   =\sum_{j=1}^{\ell_\infty}a'_j+\sum_{j=1}^{m_\infty}c_j.
\]
In the $r$th and $s$th root theory, the positive contacts have ages $a_i/r$ and $a'_j/s$, while the negative contacts have ages $1+b_i/r$ and $1+c_j/s$.  The orbifold virtual-dimension formula therefore gives
\begin{align}
\label{eq:virtual-dimension}
 \operatorname{vdim}_{\CC}
 ={}&2g-2+N+\ell_0+m_0+\ell_\infty+m_\infty
     +\frac{d_{C_0}}{r}+\frac{d_{C_0}}{s}
 \notag\\
 &-\sum_{i=1}^{\ell_0}\frac{a_i}{r}
  -\sum_{i=1}^{m_0}\left(1+\frac{b_i}{r}\right)
  -\sum_{j=1}^{\ell_\infty}\frac{a'_j}{s}
  -\sum_{j=1}^{m_\infty}\left(1+\frac{c_j}{s}\right)
 \notag\\
 ={}&2g-2+N+\ell_0+\ell_\infty.
\end{align}
The last equality uses the two displayed identities for $d_{C_0}$: the integral age contributions of the negative markings cancel the $m_0+m_\infty$ marking terms, and the fractional age contributions cancel the two root-degree corrections.  Stationarity therefore imposes
\begin{equation}
\label{eq:stationary-dimension}
  \sum_{i=1}^{N}(k_i+1)
  =2g-2+N+\ell_0+\ell_\infty,
  \qquad
  \sum_{i=1}^{N}k_i-\ell_0-\ell_\infty=2g-2.
\end{equation}
The operator factors contribute exactly $u^{\sum_i k_i-\ell_0-\ell_\infty}$, hence $u^{2g-2}$.  Summing \eqref{eq:stationary-dimension} over connected components gives the same statement with $2g-2=-\chi(C)$ in the disconnected theory.  Setting $u=1$ recovers the genus-suppressed identity.
\end{proof}

\begin{remark}
In what follows, matrix elements with unequal signed degrees, or with a common signed degree $d<0$, vanish by the energy grading.  At signed degree zero, only the empty-profile stationary vacuum is identified enumeratively.  Other signed-degree-zero coefficients of the master tau function are retained as auxiliary algebraic sectors and may be nonzero.  Unless explicitly called the vacuum sector, every relative-invariant coefficient below is understood to have positive signed degree $d>0$.
\end{remark}

\section{Finite nilpotent operators for chosen negative contacts}

Fix once and for all a finite ordered list of chosen negative contacts at $0$,
\[
    \mathbf b=(b_1,\ldots,b_m),\qquad b_i<0,
\]
and a finite ordered list at $\infty$,
\[
    \mathbf c=(c_1,\ldots,c_{m'}),\qquad c_j<0.
\]
Introduce commuting central nilpotent variables
\[
  \eps_1,\ldots,\eps_m,
  \qquad
  \del_1,\ldots,\del_{m'},
  \qquad
  \eps_i^2=\del_j^2=0.
\]
They commute with one another, with the scalar parameters, and with every infinite-wedge operator.  For $B\subseteq[m]$ write $\eps_B=\prod_{i\in B}\eps_i$ and $b_B=\sum_{i\in B}b_i$; define $N_B$ by \eqref{eq:block-weight}.  For $C\subseteq[m']$ use the analogous notation $\del_C,c_C,N_C$.

\begin{definition}[Chosen negative-contact operators]
\label{def:N-chosen}
Define
\begin{equation}
\label{eq:NL-chosen}
 \cN^L_{\mathbf b}(Q;\eps)
 =\prod_{i=1}^{m}
   \exp\left(
     \sum_{\substack{B\subseteq[m]\,\mathrm{nonempty}\\ \min B=i}}
       Q^{b_B/2}N_B\eps_B\cE_{b_B}[0]
   \right),
\end{equation}
where the product is ordered increasingly in $i$.  Define
\begin{equation}
\label{eq:NR-chosen}
 \cN^R_{\mathbf c}(Q;\del)
 =\overleftarrow{\prod}_{j=1}^{m'}
   \exp\left(
     \sum_{\substack{\varnothing\ne C\subseteq[m']\\ \min C=j}}
       Q^{c_C/2}N_C\del_C\cE_{c_C}[0]^\dagger
   \right),
\end{equation}
where the arrow means that the factors are multiplied in the decreasing order $j=m',m'-1,\ldots,1$.  If $\cN^L_{\mathbf c}(Q;\del)$ denotes the operator obtained from \eqref{eq:NL-chosen} by replacing $(\mathbf b,\eps,m)$ with $(\mathbf c,\del,m')$, then the algebraic adjoint convention \eqref{eq:algebraic-adjoint} gives
\begin{equation}
\label{eq:NR-adjoint}
  \cN^R_{\mathbf c}(Q;\del)
  =\bigl(\cN^L_{\mathbf c}(Q;\del)\bigr)^\dagger.
\end{equation}
For an empty negative list, the corresponding chosen-contact operator is the identity.
\end{definition}

In what follows, we need to extract coefficients of specific monomials from expressions. Our notation for that is: $[\eps_B]F$ is the coefficient of $\eps_S$ in $F$.

\begin{lemma}[Top coefficient equals the Kumaran--Wu block]
\label{lem:top-coeff}
For every subset $S\subseteq[m]$,
\begin{equation}
\label{eq:subset-coeff-left}
 [\eps_S]\cN^L_{\mathbf b}(Q;\eps)
 =Q^{b_S/2}\sum_{P\in\Part(S)}
   \left(\prod_{B\in P}N_B\right)
   \prod_{B\in P}^{\longrightarrow}\cE_{b_B}[0].
\end{equation}
In particular,
\begin{equation}
\label{eq:top-coeff-left}
 [\eps_1\cdots\eps_m]\cN^L_{\mathbf b}(Q;\eps)
 =Q^{(b_1+\cdots+b_m)/2}\mathcal B_{\mathbf b}.
\end{equation}
For every subset $T\subseteq[m']$,
\begin{equation}
\label{eq:subset-coeff-right}
 [\del_T]\cN^R_{\mathbf c}(Q;\del)
 =Q^{c_T/2}\sum_{P'\in\Part(T)}
   \left(\prod_{C\in P'}N_C\right)
   \prod_{C\in P'}^{\longleftarrow}\cE_{c_C}[0]^\dagger.
\end{equation}
In particular,
\begin{equation}
\label{eq:top-coeff-right}
 [\del_1\cdots\del_{m'}]\cN^R_{\mathbf c}(Q;\del)
 =Q^{(c_1+\cdots+c_{m'})/2}
  \mathcal B^\dagger_{\mathbf c}.
\end{equation}
\end{lemma}

\begin{proof}
For fixed $i$, every monomial in the exponent of the $i$th factor of \eqref{eq:NL-chosen} contains $\eps_i$.  Hence that exponent squares to zero and its exponential equals one plus the exponent.  A nonzero contribution to $[\eps_S]$ is therefore obtained by choosing disjoint nonempty subsets $B\subseteq S$ whose union is $S$, namely a set partition $P$ of $S$.  The factor indexed by $\min B$ contributes exactly $Q^{b_B/2}N_B\cE_{b_B}[0]$, and the outer product orders the selected blocks by increasing minima.  This proves \eqref{eq:subset-coeff-left} and \eqref{eq:top-coeff-left}.

The same square-zero argument applies to each exponent in \eqref{eq:NR-chosen}.  Its outer product is decreasing in the minima, so the selected adjoint factors occur in exactly the reverse of the order used to define $\mathcal B_{\mathbf c}$.  This gives \eqref{eq:subset-coeff-right} and \eqref{eq:top-coeff-right}.
\end{proof}

\begin{lemma}[Group-like property]
\label{lem:group-like}
The operators $\cN^L_{\mathbf b}(Q;\eps)$ and $\cN^R_{\mathbf c}(Q;\del)$ are group-like elements of the completed infinite-wedge group.  The descendant operator
\begin{equation}
\label{eq:D-def}
  \cD^u(\mathbf x)=\exp\left(\sum_{k\ge0}x_ku^k\cE_0[k]\right)
\end{equation}
is group-like.  Hence
\begin{equation}
\label{eq:M-def}
  \cM^u_{\mathbf b,\mathbf c}(Q;\mathbf x;\eps,\del)
  =\cN^L_{\mathbf b}(Q;\eps)\,\cD^u(\mathbf x)\,\cN^R_{\mathbf c}(Q;\del)
\end{equation}
is group-like.
\end{lemma}

\begin{proof}
Each $\cE_a[r]$ is the normally ordered infinite-wedge action of a one-particle matrix, up to the standard scalar central term in the diagonal case, and hence lies in the completed Lie algebra of the infinite-wedge group.  Every factor in \eqref{eq:NL-chosen} and \eqref{eq:NR-chosen} is an exponential of such a Lie algebra element with a central nilpotent coefficient, so it is group-like after base change to the finite Artinian nilpotent coefficient ring.  The operator \eqref{eq:D-def} is an exponential of commuting diagonal Lie algebra elements, and products of group-like elements are group-like.  The fermionic bilinear relation, and equivalently the Pl\"ucker relations among matrix elements, are polynomial; consequently they remain valid under this finite Artinian base change.
\end{proof}

\section{The chosen-contact master tau function}

Let
\begin{equation}
\label{eq:Gamma-def}
  \Gamma_+(\mathbf P)=\exp\left(\sum_{k\ge1}\frac{P_k}{k}\alpha_k\right),
  \qquad
  \Gamma_-(\mathbf R)=\exp\left(\sum_{k\ge1}\frac{R_k}{k}\alpha_{-k}\right),
\end{equation}
where $\mathbf P=(P_1, P_2,...)$ and $\mathbf R=(R_1,R_2,...)$ are independent hierarchy times.  For the chosen negative lists $\mathbf b,\mathbf c$, first define the unrescaled master series
\begin{equation}
\label{eq:unrescaled-master-tau}
 \widetilde\Theta_n^{\mathbf b,\mathbf c}(\mathbf P,\mathbf R)
 =\left\langle n\right\rvert
 \Gamma_+(\mathbf P)\,
 \cM^u_{\mathbf b,\mathbf c}(Q;\mathbf x;\eps,\del)\,
 \Gamma_-(\mathbf R)
 \left\lvert n\right\rangle.
\end{equation}
Equivalently, \eqref{eq:unrescaled-master-tau} is the charge-zero vacuum expectation obtained by placing $T^{-n}$ on the left and $T^n$ on the right.  Only after the hierarchy has been formulated in the independent times do we make the continuous specialization
\begin{equation}
\label{eq:PR}
  P_k=\frac{Q^{k/2}}{u}p_k,
  \qquad
  R_k=\frac{Q^{k/2}}{u}q_k,
\end{equation}
and set
\begin{equation}
\label{eq:master-tau}
  \Theta_n^{\mathbf b,\mathbf c}(\mathbf p,\mathbf q)
  :=\left.\widetilde\Theta_n^{\mathbf b,\mathbf c}(\mathbf P,\mathbf R)
  \right|_{\eqref{eq:PR}}.
\end{equation}
The weighted completions in which these series and the specialization are defined are specified in Remark~\ref{sec:completion}.

We define the unreduced empty-profile degree-zero convention used in this note by the vacuum eigenvalues of the diagonal operators:
\begin{equation}
\label{eq:D-vacuum}
  \langle\cD^u(\mathbf x)\rangle
  =\exp\left(\sum_{k\ge0}x_ku^k[z^{k+1}]\frac{1}{\varsig(z)}\right).
\end{equation}
Thus \eqref{eq:D-vacuum}, rather than a reduced normalization, is the definition of the empty-profile degree-zero sector here; its constant term in the descendant variables is $1$.  Dividing by this scalar would not affect the bilinear hierarchy in the profile variables, but it would change the disconnected enumerative convention, so we retain the unreduced factor.

Positive partitions are allowed to be empty, and we set
\[
  p_\varnothing=q_\varnothing=1.
\]
For a positive partition $\lambda$ we write $p_\lambda=\prod_i p_{\lambda_i}$, and similarly $q_\rho=\prod_jq_{\rho_j}$.  For $S\subseteq[m]$ and $T\subseteq[m']$ we set $b_S=\sum_{i\in S}b_i$ and $c_T=\sum_{j\in T}c_j$, with empty sums equal to $0$.

\begin{proposition}[Coefficient identification]
\label{prop:coeff-id}
Let $S\subseteq[m]$ and $T\subseteq[m']$.  Let $\lambda,\rho$ be positive partitions and set
\begin{equation}
\label{eq:degree-ST}
  d_0=|\lambda|+b_S,
  \qquad
  d_\infty=|\rho|+c_T.
\end{equation}
If $d_0\ne d_\infty$, or if $d_0=d_\infty<0$, then
\[
 [p_\lambda q_\rho\eps_S\del_T]\Theta_0^{\mathbf b,\mathbf c}=0.
\]
If $d=d_0=d_\infty>0$, then
\begin{align}
\label{eq:coeff-id}
&[p_\lambda q_\rho\eps_S\del_T]\Theta_0^{\mathbf b,\mathbf c}
=
\frac{Q^d}{|\Aut(\lambda)|\,|\Aut(\rho)|}
\sum_g u^{2g-2}
\left\langle \lambda,\mathbf b_S\,\middle|\,
  \exp\left(\sum_{k\ge0}x_k\tau_k(\omega)\right)
\,\middle|\,\rho,\mathbf c_T\right\rangle^{\bullet}_{g,d}.
\end{align}
Here $\mathbf b_S$ and $\mathbf c_T$ are the sublists in the inherited order.  If $d=0$, the only degree-zero normalization specified here is the empty-profile vacuum sector
\begin{equation}
\label{eq:empty-vacuum}
 [p_\emptyset q_\emptyset\eps_\emptyset\del_\emptyset]\Theta_0^{\mathbf b,\mathbf c}
 =\langle\cD^u(\mathbf x)\rangle
 =\exp\left(\sum_{k\ge0}x_ku^k[z^{k+1}]\frac1{\varsig(z)}\right),
\end{equation}
whose constant term in $\mathbf x$ is $1$.  Other signed-degree-zero coefficients are part of the algebraic master series and can be nonzero, but no nonempty such coefficient is identified here with a relative invariant.
\end{proposition}

\begin{proof}
The standard expansion of the Heisenberg vertex operators gives
\begin{equation}
\label{eq:Gamma-exp}
 \Gamma_+(\mathbf P)=\sum_\lambda\frac{P_\lambda}{z(\lambda)}\prod_i\alpha_{\lambda_i},
 \qquad
 \Gamma_-(\mathbf R)=\sum_\rho\frac{R_\rho}{z(\rho)}\prod_j\alpha_{-\rho_j},
\end{equation}
where $z(\lambda)=|\Aut(\lambda)|\prod_i\lambda_i$.  With \eqref{eq:PR}, the coefficient of $p_\lambda$ in $\Gamma_+$ is
\[
  \frac{Q^{|\lambda|/2}}{|\Aut(\lambda)|}
  \prod_i\frac{\alpha_{\lambda_i}}{u\lambda_i},
\]
and similarly for the right profile.  Lemma~\ref{lem:top-coeff} identifies the negative-contact coefficients with
\[
 Q^{b_S/2}\mathcal B_{\mathbf b_S},
 \qquad
 Q^{c_T/2}\mathcal B^\dagger_{\mathbf c_T}.
\]
Expanding $\cD^u(\mathbf x)$ gives the exponential of stationary insertions.  Lemma~\ref{lem:u-restore} supplies the homogeneous genus variable.  Substitution in the whole-adjoint tube formula \eqref{eq:KW} gives \eqref{eq:coeff-id}.  The total $Q$-power is
\[
  \frac12(|\lambda|+b_S)+\frac12(|\rho|+c_T),
\]
which equals $d$ when the signed degrees agree.

For the vanishing statements, use the energy operator \eqref{eq:energy-operator}.  By \eqref{eq:energy-commutator}, every monomial contributing to the selected matrix element is homogeneous of energy degree
\begin{equation}
\label{eq:energy-degree}
  -|\lambda|-b_S+c_T+|\rho|=d_\infty-d_0.
\end{equation}
Because $H\lvert0\rangle=0=\langle0\rvert H$, a vacuum matrix element of nonzero energy degree vanishes, proving the assertion when $d_0\ne d_\infty$.  If the common signed degree $d$ is negative, then each term in
$\mathcal B^\dagger_{\mathbf c_T}\prod_j\alpha_{-\rho_j}\lvert0\rangle$
would have energy $|\rho|+c_T=d<0$.  The charge-zero energy spectrum is nonnegative, so this vector is zero.  The empty signed-degree-zero case is exactly \eqref{eq:empty-vacuum}; the argument deliberately does not force other degree-zero terms to vanish.
\end{proof}

\begin{remark}[A non-vacuum signed-degree-zero coefficient]
\label{rem:auxiliary-degree-zero}
The auxiliary degree-zero sector is genuinely present.  Take one left negative contact $\mathbf b=(-2)$, no right negative contacts, $\mathbf x=0$, left profile $\lambda=(1,1)$, and empty right profile.  Coefficient extraction from \eqref{eq:E-comm} gives
\begin{equation}
\label{eq:degree-zero-commutator}
  [\alpha_1,\cE_{-2}[0]]=\alpha_{-1}.
\end{equation}
Since $\alpha_1\lvert0\rangle=0$, it follows that
\[
  \left\langle\alpha_1^2\cE_{-2}[0]\right\rangle
  =\left\langle\alpha_1\alpha_{-1}\right\rangle=1.
\]
The definitions therefore give
\begin{equation}
\label{eq:auxiliary-degree-zero-example}
 \left.[p_1^2\eps_1]\Theta_0^{(-2),\varnothing}\right|_{\mathbf x=0}
 =\frac12\left(\frac{Q^{1/2}}u\right)^2Q^{-1}
  \left\langle\alpha_1^2\cE_{-2}[0]\right\rangle
 =\frac1{2u^2}\ne0.
\end{equation}
This is a non-vacuum $Q^0$ coefficient.  It is retained as an auxiliary algebraic coefficient and is not assigned an enumerative meaning here.  In particular, the positive-degree coefficient equations need not be closed on the enumeratively identified sectors: products can involve \eqref{eq:auxiliary-degree-zero-example}, and the explicit factor $Q$ in the lowest Toda equation makes degree-one relations depend on auxiliary $Q^0$ lattice coefficients.  Simply declaring all such coefficients to be zero would generally destroy the group-like origin of the series and is not justified by the bilinear identities.
\end{remark}

\begin{corollary}[Ordered chosen positive contacts]
\label{cor:chosen-ordered}
Choose ordered positive lists
\[
   \mathbf a=(a_1,\ldots,a_\ell),\qquad
   \mathbf a'=(a'_1,\ldots,a'_{\ell'}),
\]
with $a_i,a'_j>0$.  Introduce labelled variables $s_1,\ldots,s_\ell$ and $t_1,\ldots,t_{\ell'}$ and specialize
\begin{equation}
\label{eq:label-specialization}
  p_k=\sum_{i:a_i=k}s_i,
  \qquad
  q_k=\sum_{j:a'_j=k}t_j.
\end{equation}
Then
\begin{align}
\label{eq:ordered-coeff}
&[s_1\cdots s_\ell t_1\cdots t_{\ell'}\eps_1\cdots\eps_m\del_1\cdots\del_{m'}]
\Theta_0^{\mathbf b,\mathbf c}\big|_{\eqref{eq:label-specialization}}
\notag\\
&\quad=Q^d
\sum_g u^{2g-2}
\left\langle
  a_1,\ldots,a_\ell,b_1,\ldots,b_m
  \,\middle|\,
  \exp\left(\sum_{k\ge0}x_k\tau_k(\omega)\right)
  \,\middle|\,
  a'_1,\ldots,a'_{\ell'},c_1,\ldots,c_{m'}
\right\rangle^{\bullet}_{g,d},
\end{align}
where $d=\sum_i a_i+\sum_i b_i=\sum_j a'_j+\sum_j c_j>0$.  Thus any prescribed finite ordered relative insertions of arbitrary contact orders and positive signed degree are recovered as a coefficient of the chosen-contact master tau sequence.  Taking no contacts and no positive profiles recovers the distinguished degree-zero stationary vacuum; Remark~\ref{rem:auxiliary-degree-zero} explains why it is not the only algebraic $Q^0$ coefficient.
\end{corollary}

\begin{proof}
The coefficient of $s_1\cdots s_\ell$ in $p_\lambda/|\Aut(\lambda)|$ after \eqref{eq:label-specialization} is $1$ exactly for the unordered partition underlying the ordered list $\mathbf a$, and is zero otherwise.  The same applies on the right.  Apply Proposition~\ref{prop:coeff-id} with $S=[m]$ and $T=[m']$.
\end{proof}

\begin{remark}[Formal completion]\label{sec:completion}

All constructions are finite in the nilpotent directions.  Set
\begin{equation}
\label{eq:nilpotent-base-ring}
 \mathbb K_{\mathbf b,\mathbf c}
 =\CC((u))((Q^{1/2}))\otimes_{\CC}
 \frac{\CC[\eps_1,\ldots,\eps_m,\del_1,\ldots,\del_{m'}]}
 {(\eps_1^2,\ldots,\eps_m^2,\del_1^2,\ldots,\del_{m'}^2)}.
\end{equation}
The nilpotent generators are central and commuting.  Consider the polynomial algebra over \eqref{eq:nilpotent-base-ring} in the countably many independent variables $P_k,R_k$ $(k\ge1)$ and $x_j$ $(j\ge0)$, graded by
\begin{equation}
\label{eq:weights}
 \operatorname{wt}(P_k)=\operatorname{wt}(R_k)=k,
 \qquad
 \operatorname{wt}(x_j)=j+1.
\end{equation}
If $\mathcal A_d^{P,R}$ denotes its homogeneous part of weight $d$, define the weighted completion
\begin{equation}
\label{eq:independent-completion}
 \widehat{\mathcal A}^{P,R}_{\mathbf b,\mathbf c}
 :=\prod_{d\ge0}\mathcal A_d^{P,R}.
\end{equation}
Each homogeneous part involves only finitely many variables and monomials.  The series $\widetilde\Theta_n$ belongs to \eqref{eq:independent-completion}.  When primed times occur in a bilinear identity, we use the analogous completed tensor product with a second copy of the time variables.

The formal substitutions in \eqref{eq:hirota} are interpreted coefficientwise in this weighted completion.  For a fixed external weight and a fixed power of $z$, only finitely many terms can contribute, because the index of each Miwa shift matches the weight in \eqref{eq:weights}.  Thus the coefficient extraction $\resz=[z^{-1}]$ is well defined.  This is the formal-residue convention of \eqref{eq:formal-residue}, not a geometric residue at either $0$ or $\infty$.

Define $\widehat{\mathcal A}^{p,q}_{\mathbf b,\mathbf c}$ in the same way, with
$\operatorname{wt}(p_k)=\operatorname{wt}(q_k)=k$ and the same weights for the $x_j$.  The assignment
\begin{equation}
\label{eq:continuous-specialization}
 P_k\longmapsto \frac{Q^{k/2}}u p_k,
 \qquad
 R_k\longmapsto \frac{Q^{k/2}}u q_k
\end{equation}
preserves weight and therefore extends continuously from \eqref{eq:independent-completion} to $\widehat{\mathcal A}^{p,q}_{\mathbf b,\mathbf c}$.  This is why the hierarchy is first written for $\widetilde\Theta_n$ and only then specialized to $\Theta_n$: the Miwa shifts need not be interpreted as endomorphisms of the $p,q$ ring.  Negative contacts contribute only finitely many Laurent powers of $Q^{1/2}$ inside $\cN^L$ and $\cN^R$; after a positive profile is selected and the signed-degree condition is imposed, every enumeratively identified coefficient carries the factor $Q^d$ with $d>0$ (or the distinguished vacuum factor $Q^0$), while other $Q^0$ terms remain auxiliary.
\end{remark}

\section{The hierarchy}

\subsection{Kyoto/Pl\"ucker theorem}

We use the following standard form of the Kyoto construction.  If $M$ is a charge-preserving group-like element in the completed infinite-wedge representation, then
\begin{equation}
\label{eq:tau-M}
  \tau_n^M(\mathbf P,\mathbf R)=
  \left\langle n\right\rvert
  \Gamma_+(\mathbf P)M\Gamma_-(\mathbf R)
  \left\lvert n\right\rangle
\end{equation}
is a tau sequence of the Ueno--Takasaki $2$-Toda hierarchy.  We use one formal residue convention throughout:
\begin{equation}
\label{eq:formal-residue}
  \resz f(z)\,dz:=[z^{-1}]f(z).
\end{equation}
Thus no geometric sign convention at $z=\infty$ is involved.  In the weighted completion of Remark~\ref{sec:completion}, the bilinear identity is
\begin{align}
\label{eq:hirota}
&\resz
 z^{n-n'}e^{\xi(\mathbf P-\mathbf P',z)}
 \tau_n(\mathbf P-[z^{-1}],\mathbf R)
 \tau_{n'}(\mathbf P'+[z^{-1}],\mathbf R')\,dz
\notag\\
&\quad=
\resz
 z^{n-n'}e^{\xi(\mathbf R'-\mathbf R,z^{-1})}
 \tau_{n+1}(\mathbf P,\mathbf R+[z])
 \tau_{n'-1}(\mathbf P',\mathbf R'-[z])\,dz,
\end{align}
where
\[
  \xi(\mathbf P,z)=\sum_{k\ge1}\frac{P_kz^k}{k},
  \quad
  [z^{-1}]=(z^{-1},z^{-2},z^{-3},\ldots),
  \quad
  [z]=(z,z^2,z^3,\ldots).
\]
The signs on the right of \eqref{eq:hirota} correspond to the plus-sign convention for $\Gamma_-(\mathbf R)$ in \eqref{eq:Gamma-def}.  As a direct check, when $M=1$ one has
$\tau_n=\exp(\sum_{k\ge1}P_kR_k/k)$, and substitution into \eqref{eq:hirota} gives identical formal Laurent integrands on the two sides.  The lowest equation is
\begin{equation}
\label{eq:standard-lowest}
   \tau_n\tau_{n,P_1R_1}-\tau_{n,P_1}\tau_{n,R_1}=\tau_{n+1}\tau_{n-1}.
\end{equation}
This is the first Pl\"ucker relation in infinite-wedge form; see Okounkov--Pandharipande \cite[Sec.~4]{OPcompleted}, \cite[Sec.~4.1]{OPequiv}, Johnson \cite[Ch.~I and Ch.~VI]{JohnsonThesis}, and Ueno--Takasaki \cite{UT}.

\begin{theorem}[Chosen-contact $2$-Toda hierarchy]
\label{thm:main}
For every pair of finite ordered negative-contact lists $\mathbf b,\mathbf c$, the unrescaled sequence
$\{\widetilde\Theta_n^{\mathbf b,\mathbf c}\}_{n\in\ZZ}$ defined by \eqref{eq:unrescaled-master-tau} satisfies the full $2$-Toda hierarchy in the independent times $\mathbf P,\mathbf R$, including \eqref{eq:hirota}.  After the continuous specialization \eqref{eq:PR}, the lowest equation becomes
\begin{equation}
\label{eq:lowest-main}
 u^2\left(
 \Theta_n\Theta_{n,p_1q_1}-\Theta_{n,p_1}\Theta_{n,q_1}
 \right)
 =Q\Theta_{n+1}\Theta_{n-1}.
\end{equation}
The positive signed-degree coefficients of $\Theta_0$ identified in Proposition~\ref{prop:coeff-id} and Corollary~\ref{cor:chosen-ordered} therefore occur inside a coupled integrable master system.  The system also contains auxiliary signed-degree-zero coefficients and auxiliary lattice components, and the equations are not asserted to close on the enumeratively identified coefficients alone.
\end{theorem}

\begin{proof}
By Lemma~\ref{lem:group-like}, the middle operator $\cM^u_{\mathbf b,\mathbf c}$ is group-like over the finite Artinian nilpotent base.  The polynomial fermionic bilinear relation therefore applies after base change and gives \eqref{eq:hirota} for \eqref{eq:unrescaled-master-tau}.  The first Pl\"ucker relation gives \eqref{eq:standard-lowest}.  Since $P_1=Q^{1/2}p_1/u$ and $R_1=Q^{1/2}q_1/u$, one has
\[
  \partial_{P_1}=\frac{u}{Q^{1/2}}\partial_{p_1},
  \qquad
  \partial_{R_1}=\frac{u}{Q^{1/2}}\partial_{q_1},
\]
and the continuous specialization of \eqref{eq:standard-lowest} is \eqref{eq:lowest-main}.  Proposition~\ref{prop:coeff-id} supplies the positive-degree enumerative identification, while Remark~\ref{rem:auxiliary-degree-zero} records the additional algebraic sector.
\end{proof}

\begin{remark}[The lattice components are auxiliary]
\label{rem:auxiliary-n}
Only the positive-degree coefficients of the component $n=0$, together with its empty-profile vacuum, are identified above with relative Gromov--Witten theory.  Even at $n=0$ there are auxiliary signed-degree-zero coefficients, and the components $n\ne0$ are additional auxiliary lattice shifts supplied by the Kyoto construction.  In the ordinary stationary tube without negative contacts, such shifts can often be reinterpreted through a string-direction translation.  With negative contacts, however,
\[
  T^{-n}\cE_b[0]T^n=\cE_b[0]+n\cE_b[-1]=\cE_b[0]+n\alpha_b,
  \qquad b<0,
\]
so the lattice shift changes the negative-contact block operators.  No enumerative interpretation of these shifted negative-contact components is asserted here.  The hierarchy is therefore a master-tau system whose enumerative coefficients lie in the $n=0$ component and whose equations couple them both to auxiliary coefficients at $n=0$ and to the auxiliary components $n\pm1$.
\end{remark}

\subsection{Explicit coefficient equations}

For $S\subseteq[m]$ and $T\subseteq[m']$ define
\begin{equation}
\label{eq:Theta-ST}
 \widetilde\Theta_{n;S,T}
   =[\eps_S\del_T]\widetilde\Theta_n^{\mathbf b,\mathbf c},
 \qquad
 \Theta_{n;S,T}
   =[\eps_S\del_T]\Theta_n^{\mathbf b,\mathbf c}.
\end{equation}

\begin{theorem}[Lowest coupled equation for chosen negative contacts]
\label{thm:coefficient-lowest}
For all $S\subseteq[m]$ and $T\subseteq[m']$,
\begin{align}
\label{eq:coefficient-lowest}
&u^2
\sum_{\substack{A\subseteq S\\ B\subseteq T}}
\Big(
 \Theta_{n;A,B}\,\partial_{p_1}\partial_{q_1}\Theta_{n;S\setminus A,T\setminus B}
 -\partial_{p_1}\Theta_{n;A,B}\,\partial_{q_1}\Theta_{n;S\setminus A,T\setminus B}
\Big)
=
Q\sum_{\substack{A\subseteq S\\ B\subseteq T}}
 \Theta_{n+1;A,B}\Theta_{n-1;S\setminus A,T\setminus B}.
\end{align}
The positive signed-degree coefficients of each $\Theta_{0;S,T}$ are relative invariants with chosen negative sublists $\mathbf b_S,\mathbf c_T$ and arbitrary positive profiles; after the labelled specialization \eqref{eq:label-specialization}, the chosen positive variables select the corresponding terms.  Equation~\eqref{eq:coefficient-lowest} belongs to the full master system: extracting a positive degree can involve the auxiliary $Q^0$ sectors of Remark~\ref{rem:auxiliary-degree-zero} and the auxiliary lattice components, so this is not a closed equation solely among positive-degree invariants.
\end{theorem}

\begin{proof}
Take the coefficient of $\eps_S\del_T$ in \eqref{eq:lowest-main}.  Since the nilpotent variables are commutative, the coefficient of a product is obtained by splitting the occupied left and right negative labels into two subsets, giving exactly \eqref{eq:coefficient-lowest}.
\end{proof}

The full hierarchy is equally explicit in the independent times.  Taking the coefficient of $\eps_S\del_T$ in \eqref{eq:hirota} gives
\begin{align}
\label{eq:coefficient-hirota}
&\resz
 z^{n-n'}e^{\xi(\mathbf P-\mathbf P',z)}
 \sum_{\substack{A\subseteq S\\ B\subseteq T}}
 \widetilde\Theta_{n;A,B}(\mathbf P-[z^{-1}],\mathbf R)
 \widetilde\Theta_{n';S\setminus A,T\setminus B}(\mathbf P'+[z^{-1}],\mathbf R')\,dz
\notag\\
&\quad=
\resz
 z^{n-n'}e^{\xi(\mathbf R'-\mathbf R,z^{-1})}
 \sum_{\substack{A\subseteq S\\ B\subseteq T}}
 \widetilde\Theta_{n+1;A,B}(\mathbf P,\mathbf R+[z])
 \widetilde\Theta_{n'-1;S\setminus A,T\setminus B}(\mathbf P',\mathbf R'-[z])\,dz.
\end{align}
Expansion of \eqref{eq:coefficient-hirota} in the independent times is the infinite coupled system of bilinear PDEs.  The continuous specialization \eqref{eq:PR} may then be applied to each differential equation; no Miwa shift in the $\mathbf p,\mathbf q$ coefficient ring is required.  Further extraction of $p_\lambda q_\rho$ gives recurrences for fixed chosen positive profiles.  For example,
\begin{equation}
\label{eq:coefficient-derivative-rule}
 [p_\lambda q_\rho]\partial_{p_1}\Theta_{n;S,T}
 =(m_1(\lambda)+1)[p_{\lambda\cup(1)}q_\rho]\Theta_{n;S,T},
\end{equation}
where $m_1(\lambda)$ is the number of parts of $\lambda$ equal to $1$; analogous formulae hold for $q_1$ and for higher derivatives.

\subsection{First KP equations}

The $2$-Toda hierarchy contains KP in each time family.  Put
$\widetilde F_n=\log\widetilde\Theta_n^{\mathbf b,\mathbf c}$; its logarithm is defined because the constant term is invertible in the weighted completion.  The first KP equation in the independent $\mathbf P$ variables is
\begin{equation}
\label{eq:KP-P}
 \widetilde F_{n,P_1P_1P_1P_1}+6\widetilde F_{n,P_1P_1}^2
 +12\widetilde F_{n,P_2P_2}-12\widetilde F_{n,P_1P_3}=0.
\end{equation}
For $F_n=\log\Theta_n^{\mathbf b,\mathbf c}$, the specialization \eqref{eq:PR} gives
\begin{equation}
\label{eq:KP-p}
 u^2\left(F_{n,p_1p_1p_1p_1}+6F_{n,p_1p_1}^2\right)
 +12F_{n,p_2p_2}-12F_{n,p_1p_3}=0.
\end{equation}
Similarly,
\begin{equation}
\label{eq:KP-q}
 u^2\left(F_{n,q_1q_1q_1q_1}+6F_{n,q_1q_1}^2\right)
 +12F_{n,q_2q_2}-12F_{n,q_1q_3}=0.
\end{equation}
Coefficient extraction in the nilpotent variables gives the corresponding logarithmic coupled equations wherever the relevant leading sector is invertible; the bilinear equations \eqref{eq:coefficient-lowest} and \eqref{eq:coefficient-hirota} do not require such invertibility.

\section{Checks and special cases}

\subsection{No negative contacts}

If $m=m'=0$, then $\cN^L=\cN^R=1$ and
\[
  \widetilde\Theta_n
  =\left\langle n\right\rvert
  \Gamma_+(\mathbf P)\cD^u(\mathbf x)\Gamma_-(\mathbf R)
  \left\lvert n\right\rangle.
\]
This is the stationary relative tube tau sequence of Okounkov--Pandharipande.  In their relative theory, the two Toda flows are the ramification/profile variables, not descendent variables \cite[Sec.~4]{OPcompleted}.

\subsection{One negative contact}

For a single chosen left negative contact $b<0$,
\[
  [\eps_1]\cN^L_{(b)}=Q^{b/2}\cE_b[0].
\]
Thus the one-negative-contact generating function is the coefficient of a tau sequence.  It is not asserted to be an independent tau-function; it satisfies the coupled equations obtained from Theorem~\ref{thm:coefficient-lowest}.

\subsection{Two negative contacts}

For two chosen left negative contacts $b_1,b_2<0$,
\begin{equation}
\label{eq:two-neg}
 [\eps_1\eps_2]\cN^L_{(b_1,b_2)}
 =Q^{(b_1+b_2)/2}
 \left(\cE_{b_1}[0]\cE_{b_2}[0]+b_2\cE_{b_1+b_2}[0]\right).
\end{equation}
The commutator extracted from \eqref{eq:E-comm},
$[\cE_a[0],\cE_b[0]]=(a-b)\cE_{a+b}[0]$, shows that interchanging the two labels gives the same left block.

For two right contacts, the decreasing outer order in \eqref{eq:NR-chosen} gives
\begin{align}
\label{eq:two-neg-right}
 [\del_1\del_2]\cN^R_{(c_1,c_2)}
 &=Q^{(c_1+c_2)/2}
 \left(
   \cE_{c_2}[0]^\dagger\cE_{c_1}[0]^\dagger
   +c_2\cE_{c_1+c_2}[0]^\dagger
 \right)
\notag\\
 &=Q^{(c_1+c_2)/2}
 \left(
   \cE_{c_1}[0]^\dagger\cE_{c_2}[0]^\dagger
   +c_1\cE_{c_1+c_2}[0]^\dagger
 \right).
\end{align}
The second equality follows from the adjoint commutator and makes label invariance transparent.  This is the adjoint of the entire left two-contact block, not the product obtained by adjointing its factors without reversing their order.

\subsection{Vacuum check}

Setting $\mathbf x=0$ and taking no negative contacts gives
\[
 \widetilde\Theta_0
 =\langle\Gamma_+(\mathbf P)\Gamma_-(\mathbf R)\rangle
 =\exp\left(\sum_{k\ge1}\frac{P_kR_k}{k}\right),
\]
and, after \eqref{eq:PR},
\[
 \Theta_0
 =\exp\left(\frac1{u^2}\sum_{k\ge1}\frac{Q^k}{k}p_kq_k\right).
\]
Hence
\[
  \partial_{p_1}\partial_{q_1}\log\Theta_0=\frac{Q}{u^2},
\]
which is the vacuum specialization of \eqref{eq:lowest-main}.  Substitution of $\widetilde\Theta_n=\exp(\sum P_kR_k/k)$ into the full identity \eqref{eq:hirota} also makes the two formal Laurent integrands identical, confirming the three right-flow signs.

\end{document}